\documentclass[11pt]{amsart}

\usepackage{amsmath, amsfonts, amssymb,amsthm}
\usepackage{amstext}
\usepackage{mathrsfs}

\usepackage{float,epsf,subfigure}
\usepackage[all,cmtip]{xy}
\usepackage{hyperref}
\usepackage{algorithm}
\usepackage{algorithmic}
\usepackage{mathtools}
\usepackage{float}
\usepackage{bm}
\theoremstyle{plain}
\newtheorem{theorem}{Theorem}[section]

\newtheorem{cor}[theorem]{Corollary}

\newtheorem{def-thm}[theorem]{Definition-Theorem}
\newtheorem{lemma}[theorem]{Lemma}

\newtheorem{defi}[theorem]{Definition}

\newtheorem*{tha}{Theorem A}

\theoremstyle{definition}

\begin{document}
\title[Nevanlinna's Five-Value Theorem]{Nevanlinna's Five-Value  Theorem on non-positively curved complete K\"ahler manifolds}
\author[X.-J. Dong]
{Xianjing Dong}

\address{School of Mathematical Sciences \\ Qufu Normal University \\ Qufu, Jining, Shandong, 273165, P. R. China}
\email{xjdong05@126.com}


\date{}
\maketitle \thispagestyle{empty} \setcounter{page}{1}

\begin{abstract} 
Nevanlinna's five-value theorem is well-known as a famous theorem in value distribution theory, which  asserts that two non-constant  meromorphic functions  on 
$\mathbb C$ are identical if they share five distinct values  ignoring multiplicities in $\overline{\mathbb C}.$ 
The central goal of this paper is to generalize  Nevanlinna's five-value theorem to non-compact complete K\"ahler manifolds with non-positive sectional curvature
 by means of the  theory of algebraic dependence.  
  With a certain  growth condition imposed, we show that two nonconstant meromorphic functions
on such class of manifolds are identical if they share five distinct values  ignoring multiplicities in $\overline{\mathbb C}.$

 \end{abstract} 
 
\vskip\baselineskip

\setlength\arraycolsep{2pt}
\medskip

\section{Introduction}

Nevanlinna's five-value theorem (see, e.g., \cite{Hu, Hay, Yang}) is well-known as a famous theorem in value distribution theory, which  is  stated as follows. 
\begin{tha}
Let $f_1, f_2$ be two nonconstant meromorphic functions on $\mathbb C.$ Let $a_1,\cdots,a_5$ be five distinct values in $\overline{\mathbb C}.$
If $f_1, f_2$ share $a_j$ ignoring multiplicities for $j=1,\cdots,5,$ then $f_1\equiv f_2.$
\end{tha}
Five-value theorem was investigated   by a number of  authors. For instance,  C. C. Yang \cite{Yang} weakened the condition of sharing
five values to ``partially" sharing five values; 
Li-Qiao \cite{Qiao} treated it using five small functions instead of five values; G. Valiron \cite{Vali}  promoted  it onto algebroid functions; 
W. Stoll \cite{Sto} studied some related problems  when domains are parabolic manifolds; Y. Aihara \cite{A1, A3,Ai} extended five-value  theorem to the meromorphic functions on a finite covering space of $\mathbb C^m;$ 5$IM$ problem  extends also to $3CM+1IM$ and $2CM+2IM$ problems, see,  e.g.,  Hu-Li-Yang \cite{Hu} and Yang-Yi \cite{Yang}. 
We refer the reader to other  related literature such as 
Dulock-Ru \cite{Ru}, 
 H. Fujimoto \cite{Fji},  S. Ji \cite{ji, Ji} and M. Ru \cite{ru}, etc. 

  Five-value theorem has been developed  for a long time.  However,  we don't know the five-value theorem for  meromorphic functions on a general complex manifold. 
   For example, how  about it  if domain manifolds  are  non-compact  complete K\"ahler manifolds? At present, there seems to be no results in this regard. 
  
The central goal of this paper is  studying  Nevanlinna's five-value theorem for meromorphic functions on a complete K\"ahler manifold with non-positive sectional curvature.  In doing so, the key   is 
developing  a  propagation theory of algebraic dependence of meromorphic mappings on  non-positively curved complete K\"ahler manifolds by employing   Carlson-Griffiths theory  developed  by the author \cite{Dong}. 
 
 Let us introduce the main results in this paper. Let $M$ be a non-positively curved (non-compact) complete K\"ahler manifold, $X$ be a complex projective manifold of dimension not higher than that of $M.$
Refer to \cite{Dong},  one can have the \emph{characteristic function} $T_f(r,L)$ of a meromorphic mapping $f: M\rightarrow X$ for a holomorphic line bundle $L$ over $X$  (see the details in  Section 2 in the present paper).
 
 Set $$S:=S_1\cup\cdots\cup S_q,$$ 
 where $S_1,\cdots, S_q$ are hypersurfaces  of $M$ such that  $\dim_{\mathbb C}S_i\cap S_j\leq \dim_{\mathbb C} M-2$ for $i\not=j.$
Let  $D=D_1+\cdots+D_q,$  where 
$D_1,\cdots, D_q\in |L|$  such that $D$ has simple normal crossings, here $L$ is an ample line bundle over $X.$
Now, given $l$ dominant  meromorphic mappings $f_1,\cdots,f_l: M\rightarrow X.$ Assume that
 \begin{equation*}
S_j={\rm{Supp}}f^*_iD_j,  \ \ \  1\leq i\leq l, \ \ 1\leq j\leq q.
\end{equation*}
Let $\Sigma$ be  an indecomposable hypersurface   of $X^l.$ 
Moreover, define a  $\mathbb Q$-line bundle $L_0\in {\rm{Pic}}(X)\otimes\mathbb Q$ by   
\begin{equation*}
L_0=qL\otimes\left(-\tilde\gamma lF_0\right),
\end{equation*}
where   $F_0$ is some big line bundle over $X$ and $\tilde\gamma$ is a positive  rational number depending only on $\Sigma$ and $F_0.$

To establish the five-value theorem, we first show an algebraic dependence theorem as follows. 
\begin{theorem}[=Theorem \ref{uni1}] Let $f_1,\cdots,f_l: M\rightarrow X$ be  differentiably non-degenerate  meromorphic mappings given as above satisfying the growth condition
$$ \liminf_{r\rightarrow\infty}\frac{\kappa(r)r^2}{T_{f_j}(r, L)}=0, \ \ \  j=1,\cdots, l,$$
where $\kappa$ is defined by $(\ref{ka})$ in Section $2.$ 
  Assume that $f_1,\cdots,f_l$  are $\Sigma$-related on $S.$  If $L_0\otimes K_X$ is big, then $f_1,\cdots,f_l$  are $\Sigma$-related on $M.$
\end{theorem}

Then, we  prove the following five-value theorem: 

\begin{theorem}[=Theorem \ref{thm222}]\label{dde3}
Let $f_1, f_2$ be two nonconstant meromorphic functions on $M$ satisfying the growth condition 
$$ \liminf_{r\rightarrow\infty}\frac{\kappa(r)r^2}{T_{f_j}(r, \omega_{FS})}=0, \ \ \ j=1, 2,$$ 
where  $\kappa$ is defined by $(\ref{ka})$ in Section $2.$   Let $a_1,\cdots,a_5$ be five distinct values in $\overline{\mathbb C}.$
If $f_1, f_2$ share $a_j$ ignoring multiplicities for $j=1,\cdots,5,$ then $f_1\equiv f_2.$
\end{theorem}

It is clear that $\kappa\equiv0$ for $M=\mathbb C^m.$ So, Theorem \ref{dde3} deduces  Theorem A. 

\section{Carlson-Griffiths theory and its consequences}

\subsection{Carlson-Griffiths theory}~

  Let $M$ be a  non-compact complete K\"ahler manifold of complex dimension $m,$  with Laplace-Beltrami operator $\Delta$ associated to  K\"ahler metric $g.$  
  Assume 
    that  $M$ has non-positive sectional curvature.  Fix a reference point $o\in M.$ Denote by $B(r)$ the geodesic ball in  $M,$ with center $o$ and radius $r.$
    Moreover,   denote by   $g_r(o,x)$  the positive Green function of $\Delta/2$  for $B(r),$ with a pole at $o$  satisfying  Dirichlet boundary condition.  Then, $g_r(o,x)$ defines a unique harmonic measure $d\pi_r$ on the boundary $\partial B(r),$ i.e.,  
      \begin{equation*}\label{Har}
d\pi_{r}(x)=-\frac{1}{2}\frac{\partial g_{r}(o, x)}{\partial \vec\nu}d\sigma_{r}(x), \ \ \  ^\forall x\in\partial B(r),
\end{equation*}
where  $\partial/\partial \vec\nu$ is the inward  normal derivative on $\partial B(r),$   $d\sigma_r$ is the Riemannian area element of $\partial B(r).$
   Let $Ric$  be the Ricci curvature tensor  of $M.$
Set 
\begin{equation}\label{ka}
  \kappa(r)=\frac{1}{2m-1}\inf_{x\in B(r)}\mathcal{R}(x),
\end{equation}
where $\mathcal{R}$ is the
pointwise lower bound of Ricci curvature defined by
$$\mathcal{R}(x)=\inf_{\xi\in T_{x}M, \ \|\xi\|=1} Ric(\xi,\bar{\xi}).$$

  Let $X$ be a complex projective manifold where  we put an ample Hermitian line bunlde $(L, h)$  such that its Chern form
$c_1(L,h)>0.$
  Let $f: M\rightarrow X$ be a meromorphic mapping, by which we mean that  $f$ is defined by  a  
 holomorphic mapping  $f_0:M\setminus I\rightarrow X$  such that 
 the closure  $\overline{G(f_0)}$ of  the graph $G(f_0)$ of  $f_0$  is an  analytic subset of $M\times X,$ and
the natural projection $\pi: \overline{G(f_0)}\rightarrow M$ is a proper mapping, 
 where $I$ is an analytic  subset  (called  the indeterminacy set of $f$) of $M$  satisfying  $\dim_{\mathbb C}I\leq m-2.$ 
Assume that $o\not\in I.$ Set  
 $$e_f=2m\frac{f^*c_1(L, h) \wedge\alpha^{m-1}}{\alpha^m}=-\frac{1}{2}\Delta\log(h\circ f),$$
 where 
 $$\alpha=\frac{\sqrt{-1}}{\pi}\sum_{i,j=1}^mg_{i\bar{j}}dz^i\wedge d\bar{z}^j$$
 is the  K\"ahler  form of $M.$ For a divisor $E$ on $M,$ we put
  $$\ N(r, E)=\frac{\pi^m}{(m-1)!}\int_{E\cap B(r)}g_r(o,x)\alpha^{m-1}.$$
\ \ \ \  Let $D\in|L|$ (i.e., an effective divisor $D$ which defines $L$).  The Nevanlinna's functions (\emph{characteristic function}, \emph{proximity function} and \emph{counting function} as well as  \emph{simple counting function}) are   defined by 
  \begin{eqnarray*}
T_f(r, L)&=&\frac{1}{2}\int_{B(r)}g_r(o,x)e_fdv, \\
m_f(r,D)&=&\int_{\partial B(r)}\log\frac{1}{\|s_D\circ f\|}d\pi_r, \\
N_f(r,D)&=& N(r, f^*D), \\
\overline{N}_f(r, D)&=&N(r, {\rm{Supp}}f^*D), 
 \end{eqnarray*}
respectively, where $s_D$ is  the canonical  section  associate to $D$ (i.e., a holomorphic section of $L$ with zero divisor  $D$), and  $dv$ is the Riemannian volume element of $M.$

In 2022,  the  author \cite{Dong}  gave an extension of     Carlson-Griffiths theory \cite{gri, gri1} to a non-positively curved complete K\"ahler manifold, namely, who obtained the following first main theorem and second main theorem.  
\begin{theorem}[Dong, \cite{Dong}]\label{first}  Let $M$ be a non-compact complete K\"ahler manifold, and  $X$ be a complex projective manifold. 
Let $f: M\rightarrow X$ be a meromorphic mapping such that  $f(o)\not\in {\rm{Supp}}D.$ Then 
$$m_f(r,D)+N_f(r,D)=T_f(r,L)+O(1).$$
\end{theorem}

\begin{theorem}[Dong, \cite{Dong}]\label{main}  Let $M$ be a non-compact complete K\"ahler manifold with   non-positive sectional  curvature, and  $X$ be a complex projective manifold of complex dimension not higher  than that  of $M.$
 Let $D\in|L|$ be a divisor of simple normal crossing type,  where $L$ is an ample  line bundle over $X.$ Let $f:M\rightarrow X$ be a differentiably non-degenerate meromorphic mapping. Then  for any $\delta>0,$ there exists a subset $E_\delta\subseteq(0, \infty)$ with finite Lebesgue measure such that 
  \begin{eqnarray*}
 T_f(r,L)+T_f(r, K_X) 
&\leq&  \overline N_f(r,D)+O\left(\log^+T_f(r,L)-\kappa(r)r^2+\delta\log r\right)
 \end{eqnarray*}
holds for all $r>0$ outside $E_\delta.$
\end{theorem}

Now, we  consider a defect relation. Define the \emph{defect} $\delta_f(D)$ and the \emph{simple defect} $\bar\delta_f(D)$ of $f$ with respect to $D,$ respectively   by
 \begin{eqnarray*}
\delta_f(D)&=&1-\limsup_{r\rightarrow\infty}\frac{N_f(r,D)}{T_f(r,L)}, \\
 \bar\delta_f(D)&=&1-\limsup_{r\rightarrow\infty}\frac{\overline{N}_f(r,D)}{T_f(r,L)}.
 \end{eqnarray*}
Using  the first main theorem, we obtain  $0\leq \delta_f(D)\leq\bar\delta_f(D)\leq 1.$

For any two  holomorphic line bundles $L_1, L_2$ over $X,$  define  (see \cite{gri, gri1})
$$\left[\frac{c_1(L_2)}{c_1(L_1)}\right]=\inf\left\{s\in\mathbb R: \ \omega_2<s\omega_1;  \ ^\exists\omega_1\in c_1(L_1),\  ^\exists\omega_2\in c_1(L_2) \right\},$$
where $c_1(L_j)$ is the first Chern class of $L_j$ for $j=1,2.$

\begin{cor}[Defect relation]\label{dde}  Assume the same conditions as in Theorem $\ref{main}.$  If $f$ satisfies the growth condition
$$ \liminf_{r\rightarrow\infty}\frac{\kappa(r)r^2}{T_f(r, L)}=0,$$
then
$$\delta_f(D)\leq\bar\delta_f(D)\leq  \left[\frac{c_1(K_X^*)}{c_1(L)}\right].$$
\end{cor}
\subsection{Consequences}~

A $\mathbb Q$-line bundle is  an element in ${\rm{Pic}}(M)\otimes\mathbb Q,$  where  ${\rm{Pic}}(M)$ denotes  the Picard group over $M.$
Let $F\in{\rm{Pic}}(M)\otimes\mathbb Q$ be a $\mathbb Q$-line bundle.    $F$ is said to be 
\emph{ample} (resp.  \emph{big}), if $\nu F\in {\rm{Pic}}(M)$
is ample (resp.  big) for some positive  integer $\nu.$

Let $X$ be a  complex projective manifold. Let $L$ be an ample line bundle over $X.$   For a holomorphic line bundle $F$ over $X,$  define  
$$\left[\frac{F}{L}\right]=\inf\left\{\gamma\in\mathbb Q: \gamma L\otimes F^{-1} \ \text{is big}\right\}.$$
It is easy  to see that  $[F/L]<0$ if and only if $F^{-1}$ is big. 
Let $f: M\rightarrow X$ be a meromorphic mapping,  where $M$ is a  K\"ahler manifold.  For $F\in {\rm{Pic}}(X)\otimes\mathbb Q,$  we define 
$$T_f(r,F)=\frac{1}{\nu}T_f(r,\nu F),$$
where $\nu$ is a positive integer such that $\nu F\in {\rm{Pic}}(X).$ Evidently, this is well defined. 
\begin{theorem}[Defect relation]\label{defect} 
 Let $M$ be a non-compact complete K\"ahler manifold with   non-positive sectional  curvature, and  $X$ be a complex projective manifold of complex dimension not higher  than that  of $M.$
 Let $D\in|L|$ be a divisor of simple normal crossing type,  where $L$ is an ample  line bundle over $X.$ Let $f:M\rightarrow X$ be a differentiably non-degenerate meromorphic mapping.  If $f$ satisfies the growth condition
$$ \liminf_{r\rightarrow\infty}\frac{\kappa(r)r^2}{T_f(r, L)}=0,$$
then  
 \begin{equation*}
\delta_f(D)\leq\bar\delta_f(D)\leq\left[\frac{K^{-1}_X}{L}\right].
 \end{equation*}
\end{theorem}

\begin{proof}  It follows from the definition of $[K_X^{-1}/L]$  that $([K_X^{-1}/L]+\epsilon)L\otimes K_X$ is big for any rational number  $\epsilon>0$.  Then, we obtain  
$$\left(\left[K^{-1}_X/L\right]+\epsilon\right)L\otimes K_X\geq\delta L$$
for a sufficiently small  rational number $\delta>0.$ This implies that 
$$T_f(r,K^{-1}_X)\leq \left(\left[K^{-1}_X/L\right]-\delta+\epsilon\right)T_f(r,L)+O(1).$$
By Theorem \ref{main}, we conclude that 
$$\delta_f(D)\leq\bar\delta_f(D)\leq
\left[\frac{K^{-1}_X}{L}\right].$$ 
\end{proof}
\begin{theorem}\label{cor1}  Let $M$ be a non-compact complete K\"ahler manifold, and $X$ be a complex projective manifold.    Let $f:M\rightarrow X$ be a differentiably non-degenerate meromorphic mapping. 
Assume that  $\mu F\otimes L^{-1}$ is big for some positive integer $\mu,$ where $F$ is  a big line bundle and $L$ is a holomorphic line bundle over $X.$ Then 
$$T_f(r, L)\leq \mu T_f(r, F)+O(1).$$
\end{theorem}
\begin{proof}
The bigness of $\mu F\otimes L^{-1}$  implies that there exists a nonzero holomorphic section $s\in H^0(X, \nu(\mu F\otimes L^{-1}))$ for a sufficiently large positive integer $\nu.$ By Theorem \ref{first} 
  \begin{eqnarray*}
N_f(r, (s))&\leq& T_f(r,\nu(\mu F\otimes L^{-1}))+O(1)\\
&=& \mu\nu T_f(r, F)-\nu T_f(r,L)+O(1).
  \end{eqnarray*}
  This leads to the desired inequality. 
\end{proof}

\section{Nevanlinna's five-value theorem}

\subsection{Propagation of algebraic dependence}~

Now let $M$ be a non-compact complete K\"ahler manifold with  non-positive sectional curvature.  Let $X$ be a complex projective manifold with complex dimension not higher than that of $M.$
Fix an integer $l\geq2.$   
A proper algebraic subset $\Sigma$ of $X^l$ is said to be  \emph{decomposible}, if there exist  $s$ positive integers $l_1,\cdots,l_s$ 
with $l=l_1+\cdots+l_s$ for some  integer $s\leq l$ 
and algebraic subsets $\Sigma_j\subseteq X^{l_j}$ for $1\leq j\leq s,$ such that 
$\Sigma=\Sigma_1\times\cdots\times\Sigma_s.$ If $\Sigma$ is not decomposable,  we say that $\Sigma$ is \emph{indecomposable.} 
For  $l$ meromorphic mappings $f_1,\cdots,f_l: M\rightarrow X,$
 there is  a meromorphic mapping
 $f_1\times\cdots\times f_l: M\rightarrow X^l,$ defined by
 $$(f_1\times\cdots\times f_l)(x)=(f_1(x),\cdots, f_l(x)), \ \ \  ^{\forall} x\in M\setminus \bigcup_{j=1}^lI(f_j),$$
 where $I(f_j)$ denotes the indeterminacy set of $f_j$ for $1\leq j\leq l.$
As a matter of convenience, set 
$$\tilde f=f_1\times\cdots\times f_l.$$
\begin{defi} Let $S$ be an analytic subset of  $M.$  The nonconstant meromorphic mappings $f_1,\cdots,f_l: M\rightarrow X$ are said to be algebraically dependent on $S,$ if there exists a proper indecomposable algebraic subset $\Sigma$ of $X^l$ such that $\tilde f(S)\subseteq\Sigma.$ In this case, we  say that $f_1,\cdots,f_l$ are $\Sigma$-related on $S.$
\end{defi}

  Let $L$ be an ample line bundle over $X,$ and let $D_1,\cdots,D_q\in |L|$ such that $D_1+\cdots+D_q$ has only  simple normal crossings.  
   Set 
\begin{equation*}
\mathscr G=\big\{\text{$f: M\rightarrow X$ is a differentiable non-degenerate meromorphic mapping}\big\}. 
\end{equation*}
\ \ \ \  Let $S_1,\cdots, S_q$ be hypersurfaces of $M$ such that $\dim_{\mathbb C}S_i\cap S_j\leq m-2$ for all $i\not=j.$
 Denote by
\begin{equation}
\mathscr F_\kappa=\mathscr F_\kappa\big(f\in \mathscr G; (M, \{S_j\}); (X, \{D_j\})\big)
\end{equation}
the set of all $f\in\mathscr G$ satisfying  
$$S_j={\rm{Supp}}f^*D_j, \ \ \ 1\leq j\leq q$$
and $$ \liminf_{r\rightarrow\infty}\frac{\kappa(r)r^2}{T_f(r, L)}=0.$$
  Let $\tilde L$ be a big line bundle over $X^l.$ In general, we have 
$$\tilde L\not\in \pi^*_1{\rm{Pic}}(X)\oplus\cdots\oplus\pi^*_l{\rm{Pic}}(X),$$
where $\pi_k:X^l\rightarrow X$ is the natural projection on the $k$-th factor for $1\leq k\leq l.$
Let $F_1,\cdots,F_l$   be big line bundles over $X.$ Then, it defines a line bundle over $X^l$ by
$$\tilde F=\pi^*_1F_1\otimes\cdots\otimes\pi^*_lF_l.$$
If  $\tilde L\not=\tilde F,$  we  assume that there is a  rational number $\tilde\gamma>0$ such that $$\tilde\gamma\tilde F\otimes\tilde L^{-1} \ \text{is big}.$$
  If 
$\tilde L=\tilde F,$  we shall take $\tilde\gamma=1.$ In further,    assume that there is  a line bundle $F_0\in\{F_1,\cdots,F_l\}$ such that $F_0\otimes F_j^{-1}$ is either  big or trivial for $1\leq j\leq l.$

Let $\mathscr H$ be the set of all indecomposable  hypersurfaces $\Sigma$ in $X^l$ satisfying $\Sigma={\rm{Supp}}\tilde D$ for some $\tilde D\in|\tilde L|.$  Set 
$$S=S_1\cup\cdots\cup S_q.$$

\begin{lemma}\label{lem1} Let $f_1,\cdots,f_l\in\mathscr F_\kappa.$  Assume that $\tilde f(S)\subseteq \Sigma$ and $\tilde f(M)\not\subseteq \Sigma$ for some $\Sigma\in\mathscr H.$
Then  
$$N(r, S)\leq\tilde\gamma\sum_{j=1}^lT_{f_j}(r, F_j)+O(1) \leq \tilde\gamma\sum_{j=1}^lT_{f_j}(r, F_0)+O(1).
$$
\end{lemma}
\begin{proof}  Take $\tilde D\in|\tilde L|$ such that $\Sigma={\rm{Supp}}\tilde D.$  As mentioned  earlier,  $\tilde\gamma\tilde F\otimes\tilde L^{-1}$ is big for $\tilde\gamma\not=1$ and trivial for  $\tilde\gamma=1.$   Then, by conditions with Theorem \ref{first} and Theorem \ref{cor1}, we conclude that 
  \begin{eqnarray*}
N(r, S)&\leq& T_{\tilde f}(r, \tilde L)+O(1) \\
&\leq&\tilde\gamma T_{\tilde f}(r, \tilde F)+O(1) \\
&\leq& \tilde\gamma\sum_{j=1}^l T_{f_j}(r, F_j)+O(1) \\
&\leq& \tilde\gamma\sum_{j=1}^lT_{f_j}(r, F_0)+O(1).
  \end{eqnarray*}
  The proof is completed. 
\end{proof}

Define  
\begin{equation}\label{L0}
L_0=qL\otimes\left(-\tilde\gamma lF_0\right).
\end{equation}
Again, set
$$T(r, Q)=\sum_{j=1}^lT_{f_j}(r, Q)$$
for an arbitrary  $\mathbb Q$-line bundle $Q\in{\rm{Pic}}(X)\otimes\mathbb Q.$
\begin{theorem}\label{uni1} Let $f_1,\cdots,f_l\in\mathscr F_\kappa.$ Assume that $f_1,\cdots,f_l$  are $\Sigma$-related on $S$ for some $\Sigma\in \mathscr H.$ If $L_0\otimes K_X$ is big, then $f_1,\cdots,f_l$  are $\Sigma$-related on $M.$
\end{theorem}

\begin{proof}
It suffices to prove $\tilde f(M)\subseteq\Sigma.$ Otherwise,  we assume that $\tilde f(M)\not\subseteq\Sigma.$ According to Theorem \ref{main},  for $i=1,\cdots, l$ and $j=1,\cdots, q$
$$T_{f_i}(r, L)+T_{f_i}(r, K_X) 
\leq \overline{N}_{f_i}(r, D_j)+o\big{(}T_{f_i}(r,L)\big{)},
$$
which follows  from $S_j={\rm{Supp}}f_i^*D_j$ with $1\leq i\leq l$  and $1\leq j\leq q$ that 
 \begin{eqnarray*}
qT_{f_i}(r, L)+T_{f_i}(r, K_X) 
&\leq& N(r, S)+o\big{(}T_{f_i}(r,L)\big{)}.
  \end{eqnarray*}
Using Lemma \ref{lem1}, then 
  \begin{eqnarray*}
qT_{f_i}(r, L)+T_{f_i}(r, K_X) 
&\leq& \tilde\gamma\sum_{i=1}^lT_{f_i}(r, F_0)+o\big{(}T_{f_i}(r,L)\big{)} \\
&=& \tilde\gamma T(r, F_0)+o\big{(}T_{f_i}(r,L)\big{)}.
  \end{eqnarray*}
  Thus, we get 
   \begin{eqnarray*}
qT(r, L)+T(r, K_X) 
&\leq&
  \tilde\gamma l T(r, F_0)+o\big{(}T(r,L)\big{)}.
  \end{eqnarray*} 
  It yields that 
  \begin{equation}\label{opq}
  T(r, L_0)+T(r,K_X)\leq o\big{(}T(r,L)\big{)}.
  \end{equation}
  On the other hand, the bigness of $L_0\otimes K_X$ implies that there exists  a positive integer $\mu$ 
  such that $\mu(L_0\otimes K_X)\otimes L^{-1}$ is  big. By Theorem \ref{cor1}
  $$T(r,L)\leq \mu\big(T(r,L_0)+T(r,K_X)\big)+O(1),$$
  which contradicts with (\ref{opq}).  Therefore,  we have $\tilde f(M)\subseteq\Sigma.$ 
\end{proof}

Set 
$$\gamma_0=\left[\frac{L_0^{-1}\otimes K^{-1}_X}{L}\right],$$
where $L_0$ is defined by (\ref{L0}).  Note that  $L_0\otimes K_X$ is big if and only if $\gamma_0<0.$ Thus, it yields that  
\begin{cor} Let $f_1,\cdots,f_l$ be  meromorphic mappings in $\mathscr F.$ Assume that $f_1,\cdots,f_l$  are $\Sigma$-related on $S$ for some $\Sigma\in \mathscr H.$ If $\gamma_0<0,$  then $f_1,\cdots,f_l$  are $\Sigma$-related on $M.$
\end{cor}

\subsection{Five-value theorem}~ 

We use the same notations as in Section 3.1. 
  Since $X$ is projective,  there is a  holomorphic embedding $\Phi: X \hookrightarrow\mathbb P^N(\mathbb C).$
 Let $\mathscr O(1)$ be the hyperplane line bundle over $ \mathbb P^N(\mathbb C).$ Take $l=2$ and $F_1=F_2=\Phi^*\mathscr O(1),$ then  it follows that  $F_0=\Phi^*\mathscr O(1)$ and 
 $$\tilde F=\pi_1^*\left(\Phi^*\mathscr O(1)\right)\otimes \pi_2^*\left(\Phi^*\mathscr O(1)\right).$$
 Again, set $\tilde L=\tilde F,$ then $\tilde\gamma=1.$ 
 In view of (\ref{L0}), we  have
 \begin{equation}\label{L00}
 L_0=qL\otimes\left(-2\Phi^*\mathscr O(1)\right).
 \end{equation}
 \ \ \ \  Suppose  that $D_1,\cdots, D_q\in|L|$ satisfy that  $D_1+\cdots+D_q$ has only simple normal crossings. For 
 $f_0\in\mathscr F_\kappa,$ we say that  the set $\{D_j\}_{j=1}^q$ is \emph{generic} with respect to $f_0$  if 
 $$\hat M_s=f_0(M-I(f_0))\cap{\rm{Supp}}D_s\not=\emptyset$$
 for at least one $s\in\{1,\cdots,q\}.$ Now, we assume that   $\{D_j\}_{j=1}^q$ is generic with respect to $f_0.$
  Denote by  $\mathscr F_{\kappa, 0}$   the set of all meromorphic mappings $f\in\mathscr F_\kappa$ such that $f=f_0$ on the   hypersuface $S.$
 \begin{lemma}\label{t1} 
 If $L_0\otimes K_X$ is big, then $\mathscr F_{\kappa, 0}$ has only  one element. 
 \end{lemma}
 
 \begin{proof}   It suffices to show that $f\equiv f_0$ for all $f\in\mathscr F_{\kappa, 0}.$ 
 Recall that  $\Phi: X \hookrightarrow\mathbb P^N(\mathbb C)$ is a   holomorphic embedding. 
 Since  $f=f_0$ on $S,$  we have   $\Phi\circ f=\Phi\circ f_0$ on $S.$
First, we  assert  that
 $\Phi\circ f\equiv\Phi\circ f_0.$  
   Otherwise,   we may assume that $\Phi\circ f\not\equiv\Phi\circ f_0.$
 Let $\Delta$ denote the diagonal of $\mathbb P^N(\mathbb C)\times \mathbb P^N(\mathbb C).$ Put $\tilde \Phi=\Phi\times \Phi$
 and $\tilde f=f\times f_0.$ Then, it gives   a meromorphic mapping
$$\phi=\tilde \Phi\circ \tilde f:=\Phi\circ f\times \Phi\circ f_0: \  M\rightarrow \mathbb P^N(\mathbb C)\times \mathbb P^N(\mathbb C).$$
It is clear that $\phi(S)\subseteq \Delta.$ Again, define $\tilde{\mathscr O}(1):=\pi_1^*\mathscr O(1)\otimes\pi_2^*\mathscr O(1),$ which is a holomorphic line bundle over $\mathbb P^N(\mathbb C)\times \mathbb P^N(\mathbb C),$ where $\mathscr O(1)$ is the hyperplane line bundle over  $\mathbb P^N(\mathbb C).$ From the  assumption, we have  
$\tilde L=\pi_1^*\left(\Phi^*\mathscr O(1)\right)\otimes \pi_2^*\left(\Phi^*\mathscr O(1)\right).$  Since  $\Phi\circ f\not\equiv\Phi\circ f_0,$  then there exists  a  holomorphic  section $\tilde\sigma$ of $\tilde{\mathscr O}(1)$ over $\mathbb P^N(\mathbb C)\times \mathbb P^N(\mathbb C)$ such  that $\phi^*\tilde\sigma\not=0$ and $\Delta\subseteq{\rm{Supp}}(\tilde\sigma).$ 
 Take $\Sigma={\rm{Supp}}\tilde\Phi^*(\tilde\sigma),$ then we have    $\tilde f(S)\subseteq\Sigma$ and $\tilde f(M)\not\subseteq\Sigma.$  
On the other hand,  with the aid of  Theorem \ref{uni1}, the bigness of  $L_0\otimes K_X$ gives that  $\tilde f(M)\subseteq\Sigma,$ which is a contradiction.   Hence, we obtain  $\Phi\circ f\equiv\Phi\circ f_0.$   
Next, we prove $f\equiv f_0.$ Otherwise, we have $f(x_0)\not=f_0(x_0)$ for some $x_0\in M\setminus I(f_0).$  However, it contradicts with
$\Phi(f(x_0))=\Phi(f_0(x_0))$ since $\Phi$ is injective.  
 \end{proof}

 \begin{theorem}\label{thm222} Let $f_1, f_2$ be nonconstant meromorphic  functions on $M$ satisfying the growth condition
$$ \liminf_{r\rightarrow\infty}\frac{\kappa(r)r^2}{T_{f_j}(r, \omega_{FS})}=0, \ \ \ j=1,2.$$
 Let $a_1,\cdots,a_q$ be distinct values in $\overline{\mathbb C}.$    
Assume that ${\rm{Supp}}f_1^*a_j={\rm{Supp}}f_2^*a_j$ for  $j=1,\cdots,q.$ If $q\geq 5,$ then $f_1\equiv f_2.$
  \end{theorem}
 \begin{proof}
 Set $X=\mathbb P^1(\mathbb C)$ and $L=\mathscr O(1).$   Note that  $K_{\mathbb P^1(\mathbb C)}=-2\mathscr O(1),$ then 
  $$L_0\otimes K_{\mathbb P^1(\mathbb C)}=q\mathscr O(1)\otimes(-2\mathscr O(1))\otimes(-2\mathscr O(1))=(q-4)\mathscr O(1).$$
 Hence,  $L_0\otimes K_{\mathbb P^1(\mathbb C)}$ is big for  $q\geq5.$ By Lemma \ref{t1}, we prove the theorem. 
 \end{proof}
 
  \begin{cor} Let $f_1, f_2$ be nonconstant meromorphic  functions on $\mathbb C^m.$  Let $a_1,\cdots,a_q$ be distinct values in $\overline{\mathbb C}.$   
Assume that ${\rm{Supp}}f_1^*a_j={\rm{Supp}}f_2^*a_j$ for  $j=1,\cdots,q.$ If $q\geq 5,$ then $f_1\equiv f_2.$
  \end{cor}

\vskip\baselineskip

\end{document}